\documentclass[11pt]{amsart}
\usepackage{amssymb}
\usepackage{dsfont}

\newtheorem{theorem}{Theorem}[section]
\newtheorem{lemma}[theorem]{Lemma}
\newtheorem{proposition}[theorem]{Proposition}

\newtheorem{corollary}[theorem]{Corollary}

\newcommand{\Q}{\mathbb Q}

\newcommand{\R}{\mathbb R}

\newcommand{\C}{\mathbb C}

\newcommand{\on}{\operatorname}
\author{Artur Bartoszewicz}

\address{Institute of Mathematics, \L\'od\'z University of Technology, W\'olcza\'nska 215, 93-005
\L\'od\'z, Poland}
\email {arturbar@p.lodz.pl}

\author{Szymon G\l \c ab}

\address{Institute of Mathematics, \L\'od\'z University of Technology, W\'olcza\'nska 215, 93-005
\L\'od\'z, Poland}
\email {szymon.glab@p.lodz.pl}

\thanks{The authors have been supported by the Polish Ministry of Science and Higher Education Grant No.  N N201 414939 (2010--2013).}
\title[Additivity and lineability]{Additivity and lineability in vector spaces}
\date{}

\begin{document}

\begin{abstract} 
G\'amez-Merino, Munoz-Fern\'andez and Seoane-Sep\'ulveda proved that if additivity $\mathcal A(\mathcal F)>\mathfrak c$, then $\mathcal F$ is $\mathcal A(\mathcal F)$-lineable where $\mathcal F\subseteq\R^\R$. They asked if $\mathcal A(\mathcal F)>\mathfrak c$ can be weakened. We answer this question in negative. Moreover, we introduce and study the notions of homogeneous lineability number and lineability number of subsets of linear spaces.  
\end{abstract}

\maketitle

\section{Introduction}

A subset $M$ of a linear space $V$ is $\kappa$-lineable if $M\cup\{0\}$ contains a linear subspace of dimension $\kappa$ (see  \cite{APGS},  \cite{ACPSS}, \cite{AGS}, \cite{GMMFSSS}, \cite{GMMFSSContinuityMonthly}, \cite{GPPSS}). If additionally $V$ is a linear algebra, then in a similar way one can define albegrability of subsets of $V$ (see \cite{ACPSS}, \cite{APSS}, \cite{ASS}, \cite{BBG}, \cite{BGPS}, \cite{BGP}, \cite{BG}, \cite{BG1}, \cite{BQ}, \cite{GMSS}, \cite{GPMSS}, \cite{GPPSS}). If $V$ is a linear topological space, then $M\subseteq V$ is called spaceable (dense-lineable) if $M$ contains a closed infinitely dimensional subspace (dense subspace) (see \cite{BGo}, \cite{BGOC}, \cite{GKP}). The lineability problem of subsets of linear spaces of functions or sequences have been studied by many authors. The most common way of proving $\kappa$-lineability is to construct a set of cardinality $\kappa$ of linearly independent elements of $V$ and to show that any linear combination of them is in $M$.   

We will concentrate on a non-constructive method in lineability. Following the paper \cite{GMMFSS}, we will consider a connection between lineability and additivity. This method does not give a specific large linear space, but ensures that such a space exists.

Let $\mathcal F\subseteq\R^\R$. The additivity of $\mathcal F$ is defined as the following cardinal number
$$
\mathcal A(\mathcal F)=\min(\{\vert F\vert:F\subseteq\R^\R,\;\varphi+F\nsubseteq\mathcal F\text{ for every }\varphi\in\R^\R\}\cup\{(2^\mathfrak{c})^+\}).
$$
The notion of additivity was introduced by Natkaniec in \cite{N1} and then studied by several authors \cite{CM}, \cite{CMa}, \cite{CN}, \cite{CR} and \cite{N2}. 

A family $\mathcal F\subseteq\R^\R$ is called star-like if $a\mathcal F\subseteq \mathcal F$ for all $a\in\R\setminus\{0\}$. G\'amez-Merino, Munoz-Fern\'andez and Seoane-Sep\'ulveda proved the in \cite{GMMFSS} following result which connects the lineability and the additivity of star-like families. 
\begin{theorem}\label{thmGMMFSS}
Let $\mathcal F\subseteq\R^\R$ be star-like. If $\mathfrak{c}<\mathcal A(\mathcal F)\leq 2^\mathfrak{c}$, then $\mathcal F$ is $\mathcal A(\mathcal F)$-lineable. 
\end{theorem}
The authors noted that there is a star-like family $\mathcal F$ such that $\mathcal A(\mathcal F)=2$ and $\mathcal F$ is not 2-lineable. They asked if the above result is true for $2<\mathcal A(\mathcal F)\leq\mathfrak{c}$. We will answer this question in negative. This will show that Theorem \ref{thmGMMFSS} is sharp.  

Let us observe that the notion of additivity can be stated for abelian groups as follows. If $(G,+)$ is an abelian group and $\mathcal F\subseteq G$, then the additivity of $\mathcal F$ is the cardinal number
$$
\mathcal A(\mathcal F)=\min(\{\vert F\vert:F\subseteq G\text{ and }\forall\varphi\in G(\varphi+F\nsubseteq\mathcal F)\}\cup\{\vert G\vert^+\}).
$$
On the other hand, the lineability is a natural notion in vector spaces. We say that a set $\mathcal F\subseteq V$, where $V$ is a linear space, is $\lambda$-lineable if there exists a subspace $W$ of $V$ such that $W\subseteq\mathcal F\cup\{0\}$ and $\dim W=\lambda$. Now, the lineability  $\mathcal L(\mathcal F)$ is the cardinal number
$$
\mathcal L(\mathcal F)=\min\{\lambda:\mathcal F\text{ is not }\lambda\text{-lineable}\}.
$$
Clearly $\mathcal L(\mathcal F)$ is a cardinal number less or equal to $(\dim V)^+$ and it can take any value between 1 and $(\dim V)^+$ -- see Proposition \ref{PropertiesOfHomLin}.

%
%

\section{results}

The following Lemma \ref{lem1} and Theorem \ref{generalThm} are generalizations of \cite[Lemma 2.2 and Theorem 2.4]{GMMFSS} in the settings of abelian groups and vector spaces over infinite fields, respectively. Short proofs of this facts are essentially the same as those in \cite{GMMFSS}, but our presentation is more general and from the proof of Theorem \ref{generalThm} we extract a new notion of lineability, namely \emph{homogeneus lineability}. Moreover the authors of \cite{GMMFSS} claimed that Theorem \ref{thmGMMFSS} held true also in the case $\mathfrak{c}=\mathcal A(\mathcal F)$. However we will show (see Theorem \ref{2-lin} and Theorem \ref{mainThm}) that this is not true. Let us remark that all examples of families $\mathcal F\subseteq\R^\R$ discussed in \cite{GMMFSS} have additivity $\mathcal A(\mathcal F)$ greater than $\mathfrak c$, so the described mistake has almost no impact on the value of this nice paper.

\begin{lemma}\label{lem1}
Let $(G,+)$ be an abelian group. Assume that $F$ is a subgroup of $G$ and $\mathcal F\subseteq G$ is such that 
\begin{equation}\label{eq1}
2\vert F\vert<\mathcal A(\mathcal F).
\end{equation}
Then there is $g\in\mathcal F\setminus F$ with $g+F\subseteq\mathcal F$. That means actually that some coset of $F$ different from $F$ is contained in $\mathcal F$. 
\end{lemma}

\begin{proof}
Let $h\in G\setminus F$ and put $F_h=(h+F)\cup F$. Then $\vert F_h\vert=2\vert F\vert$. By \eqref{eq1} there is $g\in G$ such that $g+F_h\subseteq\mathcal F$. Thus $g+F\subseteq\mathcal F$, $(g+h)+F\subseteq\mathcal F$ and $0\in F$, and consequently $g\in\mathcal F$ and $g+h\in\mathcal F$. It is enough to show that $g\notin F$ or $g+h\notin F$. Suppose to the contrary that $g,g+h\in F$. Then $h=(g+h)-g\in F$ which is a contradiction.
 \end{proof}

Let us remark that if $\mathcal A(\mathcal F)$ is an infinite cardinal, then the condition $\vert F\vert<\mathcal A(\mathcal F)$ implies condition \eqref{eq1}. 

Assume that $V$ is a vector space, $A\subseteq V$ and $f_1,\dots,f_n\in V$. Fix the following notation $[A]=\on{span}(A)$ and $[f_1,\dots,f_n]=\on{span}(\{f_1,\dots,f_n\})$.

\begin{theorem}\label{generalThm}
Let $V$ be a vector space over a field $\mathbb{K}$ with $\omega\leq\vert\mathbb{K}\vert=\mu<\dim V$. Assume that $\mathcal F\subseteq V$ is star-like, that is $a\mathcal F\subseteq\mathcal F$ for every $a\in\mathbb K\setminus\{0\}$, and 
\begin{equation}\label{eq2}
\mu<\mathcal A(\mathcal F)\leq\dim V.
\end{equation}
Then $\mathcal F\cup\{0\}$ is $\mathcal A(\mathcal F)$-lineable, in symbols $\mathcal L(\mathcal F)>\mathcal A(\mathcal F)$. Moreover, any linear subspace $Y$ of $V$ contained in $\mathcal F$ of dimension less than $\mathcal A(\mathcal F)$ can be extended to $\mathcal A(\mathcal F)$-dimensional subspace also contained in $\mathcal F$.  
\end{theorem}

\begin{proof}
Let $Y$ be a linear subspace of $V$ with $Y\subseteq\mathcal F\cup\{0\}$. 
Let $X$ be a maximal element, with respect to inclusion, of the family
$$
\{X:Y\subseteq X\subseteq\mathcal F\cup\{0\},\text{ }X\text{ is a linear subspace of }V\}. 
$$
Suppose to the contrary that $\vert X\vert<\mathcal A(\mathcal F)$. Then by Lemma \ref{lem1} there is $g\in\mathcal F\setminus X$ with $g+X\subseteq\mathcal F$. Let $Z=[g]+X$ and take any $z\in Z\setminus X$. Then there is $x\in X$ and nonzero $a\in\mathbb{K}$ with $z=ag+x=a(g+x/a)\in\mathcal F$. Since $X\subseteq\mathcal F$ we obtain that $Z\subseteq\mathcal F$. This contradicts the maximality of $X$. Therefore $\mathcal A(\mathcal F)\leq\vert X\vert<\mathcal L(\mathcal F)$. 
\end{proof}

The assertion of Theorem \ref{generalThm} leads us to the following definition of a new cardinal function. We define the \textbf{homogeneous lineability number} of $\mathcal F$ as the following cardinal number
$$
\mathcal {HL}(\mathcal F)=\min(\{\lambda:\text{there is linear space }Y\subseteq\mathcal F\cup\{0\}\text{ with }\dim Y<\lambda
$$
$$
\text{ which cannot be extended to a linear space }X\subseteq\mathcal F\cup\{0\}\text{ with }\dim X=\lambda\}\cup\{\vert V\vert^+\}).
$$
Now the assertion of Theorem \ref{generalThm} can be stated briefly.

\begin{corollary}\label{corollary1}
Let $V$ be a vector space over a field $\mathbb{K}$ with $\omega\leq\vert\mathbb{K}\vert=\mu<\dim V$. Assume that $\mathcal F$ is star-like and $\mu<\mathcal A(\mathcal F)\leq\dim V.$ Then $\mathcal{HL}(\mathcal F)>\mathcal A(\mathcal F)$. 
\end{corollary}

Now we will show basic connections between $\mathcal{HL}(\mathcal F)$ and $\mathcal{L}(\mathcal F)$.

\begin{proposition}\label{PropertiesOfHomLin}
Let $V$ be a vector space. Then
\begin{itemize}
\item[(i)] $\mathcal{HL}(\mathcal F)\leq\mathcal{L}(\mathcal F)$ for every $\mathcal F\subseteq V$;
\item[(ii)] For every successor cardinal $\kappa\leq(\dim V)^+$ there is $\mathcal F\subseteq V$ with $\mathcal{HL}(\mathcal F)=\mathcal{L}(\mathcal F)=\kappa$;
\item[(iii)] For every $\lambda,\kappa\leq(\dim V)^+$ such that $\lambda^+<\kappa$ there is $\mathcal F\subseteq V$ with $\mathcal{HL}(\mathcal F)=\lambda^+$ and $\mathcal{L}(\mathcal F)=\kappa$;
\item[(iv)] $\mathcal{HL}(\mathcal F)$ is a successor cardinal.
\end{itemize} 
\end{proposition}
\begin{proof}
Note that the cardinal number $\mathcal{L}(\mathcal F)$ can be defined in the similar terms as it was done for $\mathcal{HL}(\mathcal F)$, namely $\mathcal{L}(\mathcal F)$ is the smallest cardinal $\lambda$ such that the trivial linear space $Y=\{0\}$ cannot be extended to a linear space $X\subseteq\mathcal F\cup\{0\}$ with $\dim V=\lambda$. Therefore $\mathcal{HL}(\mathcal F)\leq\mathcal{L}(\mathcal F)$.  

To see (ii) take any successor cardinal $\kappa\leq\vert V\vert^+$. There is $\lambda\leq\vert V\vert$ with $\lambda^+=\kappa$. Let $\mathcal F$ be a linear subspace of $V$ of dimension $\lambda$. Then $\mathcal F$ is $\lambda$-lineable but not $\lambda^+$-lineable. Thus $\mathcal{L}(\mathcal F)=\kappa$. Note that any linear subspace of $\mathcal F$ can be extended to a $\lambda$-dimensional space $\mathcal F$, but cannot be extended to a $\kappa$-dimensional space. Thus $\mathcal{HL}(\mathcal F)=\kappa$. 

Let $\on{Card}[\lambda,\kappa)=\{\nu:\lambda\leq\nu<\kappa$ and $\nu$ is a cardinal number$\}$. For every $\nu\in\on{Card}[\lambda,\kappa)$ we find  $B_\nu$ of cardinality $\nu$ such that $B=\bigcup\{B_\nu:\nu\in\on{Card}[\lambda,\kappa)\}$ is a linearly independent subset of $V$. Let $W_\nu=[B_\nu]$. Let $\mathcal F=\bigcup_{\nu\in\on{Card}[\lambda,\kappa)} W_\nu$. Clearly $\mathcal F$ is $\nu$-lineable for any $\nu<\kappa$. Take any linear space $W\subseteq\mathcal F$. Since $[\mathcal F]=[B]$, then any element $x$ of $W$ is of the form $\sum_{i=1}^k\sum_{j=1}^{n_i}\alpha_{ij}w_{ij}$ where $\alpha_{ij}\in\mathbb K$ and $w_{i1},\dots,w_{in_i}$ are distinct elements of $B_{\nu_i}$, $\lambda\leq\nu_1<\dots<\nu_{k}<\kappa$. If $x\in W_{\nu}$ for some $\nu$, then $x=\sum_{p=1}^m\alpha_pw_p$, $\alpha_p\in\mathbb K$, $w_1,\dots,w_m\in B_\nu$ are distinct. Then 
$$
\sum_{i=1}^k\sum_{j=1}^{n_i}\alpha_{ij}w_{ij}=\sum_{p=1}^m\alpha_pw_p.
$$
Now, if $\nu\notin\{n_0,\dots,\nu_k\}$, then $\alpha_{ij}=0=\alpha_p$ for every $i,j,p$. Thus $x=0$. If $\nu=\nu_l$, then $\alpha_{ij}=0$ for every $i\neq l$. Thus $x\in W_\nu$. This shows that $W$ does not contain any nontrivial linear combination of $\alpha_1x_1+\dots+\alpha_nx_n$ with $x_i\in W_{\nu_i}$ for distinct $\nu_1,\dots,\nu_n$. Therefore $W\subseteq W_\nu$ for some $\nu$. Consequently $\dim W<\kappa$, which means that $\mathcal F$ is not $\kappa$-lineable. Hence $\mathcal{L}(\mathcal F)=\kappa$. 

Take any linear space $Y\subseteq\mathcal F$ of dimension less than $\lambda$. As before we obtain that $Y$ is a subset of some $W_\nu$. Therefore $Y$ can be extended to a linear subspace of $\mathcal F$ of dimension $\lambda$. On the other hand $Y=W_\lambda$ cannot be extended to a linear space contained in $\mathcal F$ of dimension $\lambda^+$. Hence $\mathcal{HL}(\mathcal F)=\lambda^+$. 

To prove (iv) assume that $\mathcal{HL}(\mathcal F)=\kappa$. Then for any $\lambda<\kappa$ and any linear space $Y\subseteq\mathcal F\cup\{0\}$ of dimension less than $\lambda$ there is a linear space $X\supset Y$ contained in $\mathcal F\cup\{0\}$ of dimension $\lambda$. Suppose to the contrary that $\kappa$ is a limit cardinal. There are cardinals $\tau_\nu<\kappa$, $\nu<\on{cf}(\kappa)\leq\kappa$ with $\bigcup_{\nu<\on{cf}(\kappa)}\tau_\nu=\kappa$. Since $\mathcal{HL}(\mathcal F)=\kappa$, then for any linear space $Y\subseteq\mathcal F$ of dimension less than $\kappa$ we can inductively define an increasing chain $\{Y_\nu:\dim Y<\nu<\on{cf}(\kappa)\}$ of linear spaces with $\dim Y_\nu=\tau_\nu$ and $Y_\nu\subseteq\mathcal F\cup\{0\}$. Then $Y'=\bigcup Y_\nu$ is a linear space of dimension $\kappa$ such that $Y\subseteq Y'\subseteq\mathcal F$. Hence $\mathcal{HL}(\mathcal F)\geq\kappa^+$ which is a contradiction. 
\end{proof}

\begin{theorem}\label{2-lin}
Assume that $3\leq\kappa\leq\mu$, $\mathbb{K}$ is a field of cardinality $\mu$ and $V$ is a linear space over $\mathbb{K}$ with $\dim V=2^\mu$. Then there is a star-like family $\mathcal F\subseteq V$ with $\mathcal A(\mathcal F)=\kappa$ which is not 2-lineable. 
\end{theorem}

\begin{proof}
Let $\{G_\xi:\xi<2^\mu\}$ be an enumeration of all subsets of $V$ of cardinality less than $\kappa$. Let $\mathcal I\subseteq V$ be a linearly independent set of cardinality $\kappa$. Inductively for any $\xi<2^\mu$ we construct $\varphi_\xi\in V$ and $X_\xi\subseteq V$ such that 
\begin{itemize}
\item[(a)] $\varphi_\xi+f\notin Y_\xi:=[\mathcal I\cup\bigcup_{\beta<\xi}X_\beta]$ for any $f\in G_\xi$;
\item[(b)] $X_\xi=\bigcup_{f\in G_\xi}[\varphi_\xi+f]$. 
\end{itemize}  
Suppose that we have already constructed $\varphi_\xi$ and $X_\xi$ for every $\xi<\alpha$. Let $Y_\alpha=[\mathcal I\cup\bigcup_{\xi<\alpha}X_\xi]$. Since $\dim[X_\xi]<\kappa$, then $\dim Y_\alpha\leq\vert\alpha\vert\kappa+\kappa$. Thus $\vert Y_\alpha\vert<2^\mu$ and we can choose $\varphi_\alpha\notin Y_\alpha-G_\alpha$ (equivalently $\varphi_\alpha+f\notin Y_\alpha$ for every $f\in G_\alpha$). Define $X_\alpha=\bigcup_{f\in G_\alpha}[\varphi_\alpha+f]$.   

Observe that $X_\xi\cap X_{\xi'}=\{0\}$ and $Y_\xi\subseteq Y_{\xi'}$ for $\xi<\xi'$. Moreover $Y_\xi\cap X_{\xi'}=\{0\}$ and $X_\xi\cap[\mathcal I]=\{0\}$ for $\xi\leq\xi'$. Define $\mathcal F=\bigcup_{\xi<2^\mu}X_\xi$.  Take any $G\subseteq V$ with $\vert G\vert<\kappa$. There is $\xi$ with $G=G_\xi$. Then $\varphi_\xi+G_\xi\subseteq X_\xi\subseteq\mathcal F$. Therefore $\mathcal A(\mathcal F)\geq\kappa$. 

Now, we will show that for any $\varphi\in V$ there is $i\in\mathcal I$ with $\varphi+i\notin\mathcal F$. Suppose to the contrary that it is not the case, that is there is $\varphi\in V$ such that for any $i\in\mathcal I$ we have $\varphi+i\in\mathcal F$. Then there are distinct $i,i'\in\mathcal I$ with $\varphi+i,\varphi+i'\in\mathcal F$. Suppose first that $\varphi+i\in X_\xi$ and $\varphi+i'\in X_{\xi'}$ with $\xi<\xi'$. Then
$$
X_{\xi'}\ni\varphi+i'=\varphi+i+(i'-i)\in[X_\xi\cup\mathcal I]\subseteq Y_\xi. 
$$
Thus $\varphi+i'=0$. Therefore $\varphi\in[\mathcal I]$ and consequently $\varphi+\mathcal I\subseteq[\mathcal I]$. Since $i\neq i'$, then $X_\xi\ni\varphi+i=\varphi+i'+(i-i')=i-i'\in[\mathcal I]$ which contradicts the fact that $X_{\xi}\cap[\mathcal I]=\{0\}$. Hence there is $\xi$ such that $\varphi+i\in X_\xi$ for every $i\in\mathcal I$. Since $\vert G_\xi\vert<\kappa$ and $\vert\mathcal I\vert=\kappa$, there are two distinct $i,i'\in\mathcal I$ such that $\varphi+i=a(\varphi_\xi+f)$ and $\varphi+i'=a'(\varphi_\xi+f)$ for some $a,a'\in\mathbb{K}$ and $f\in G_\xi$. Thus $i-i'=(a-a')(\varphi_\xi+f)$ and therefore $\varphi_\xi+f\in[\mathcal I]$ which is a contradiction. 

Finally we obtain that $\varphi+\mathcal I\nsubseteq\mathcal F$ for every $\varphi\in V$, which means that $\mathcal A(\mathcal F)\leq\kappa$. Hence $\mathcal A(\mathcal F)=\kappa$. 

Let $U=[h,h']$ for two linearly independent elements $h,h'\in\mathcal F$. Then $h\in X_\xi$ and $h'\in X_{\xi'}$ for some $\xi$ and $\xi'$. If $\xi<\xi'$, then $h\in Y_{\xi'}$ and $h'\notin Y_{\xi'}$. Let $f\in U\setminus([h]\cup[h'])$. Then $f=ah+a'h'$ for some $a,a'\in\mathbb{K}\setminus\{0\}$. If $f\in Y_{\xi'}$, then $h'=(f-ah)/a'\in Y_{\xi'}$ which leads to contradiction. Thus $f\notin Y_{\xi'}$ which means that $U\cap Y_{\xi'}=[h]$. Since two-dimensional space $U$ cannot be covered by less than $\mu$ many sets of the form $[g]$, then $U\nsubseteq Y_{\xi'}\cup X_{\xi'}$. However $U\subseteq[Y_{\xi'}\cup X_{\xi'}]$. Therefore $U\cap X_\alpha=\{0\}$ for every $\alpha>\xi'$. Hence $U\nsubseteq\mathcal F$. 

If $h,h'\in X_\xi$, then $U\cap\bigcup_{\beta\neq\xi}X_\beta=\{0\}$ and $U\nsubseteq X_\xi$. That implies that $\mathcal F$ does not contains two-dimensional vector space. 

Finally, note that $\mathcal F$ is star-like. 
\end{proof}

The next result, which is a modification of Theorem \ref{2-lin}, shows that if $\mathcal A(\mathcal F)\leq\vert\mathbb{K}\vert$, then $\mathcal L(\mathcal F)$ can be any cardinal not greater than $(2^\mu)^+$.  

\begin{theorem}\label{mainThm}
Assume that $3\leq\kappa\leq\mu$, $\mathbb{K}$ is a field of cardinality $\mu$ and $V$ is a linear space over $\mathbb{K}$ with $\dim V=2^\mu$. Let $1<\lambda\leq(2^\mu)^+$. There is star-like family $\mathcal F\subseteq V$ such that $\kappa\leq\mathcal A(\mathcal F)\leq\kappa+1$ and $\mathcal L(\mathcal F)=\lambda$. 
\end{theorem}

\begin{proof}
Let $\{G_\xi:\xi<2^\mu\}$ be an enumeration of all subsets of $V$ of cardinality less than $\kappa$. Write $V$ as a direct sum $V_1\oplus V_2$ of two vector spaces $V_1$ and $V_2$ with $\dim V_1=\dim V_2=2^\mu$. Let $\on{Card}(\lambda)=\{\nu<\lambda:\nu\text{ is a cardinal number}\}$. As in the proof of Proposition \ref{PropertiesOfHomLin}(iii) we can find vector spaces $W_\nu\subseteq V_1$, $\nu\in\on{Card}(\lambda)$ such that $\dim W_\nu=\nu$ and the union of bases of all $W_\nu$'s forms a linearly independent set. Put $Z=\bigcup_{\nu\in\on{Card}(\lambda)}W_\nu\subseteq V_1$ and note that $\mathcal L(Z)=\lambda$. Let $\mathcal I\subseteq V_2$ be a linearly independent set of cardinality $\kappa+1$. Inductively for any $\xi<2^\mu$ we construct $\varphi_\xi\in V$ and $X_\xi\subseteq V$ such that $\varphi_\xi$ and $X_\xi$ satisfy the formulas (a) and (b) from the proof of Theorem \ref{2-lin}. Define $\mathcal F=Z\cup\bigcup_{\xi<2^\mu}X_\xi$. Then $\mathcal A(\mathcal F)\geq\kappa$. 

Now, we will show that for any $\varphi\in V$ there is $i\in\mathcal I$ with $\varphi+i\notin\mathcal F$. Suppose to the contrary that it is not the case, that is there is $\varphi\in V$ such that for any $i\in\mathcal I$ we have $\varphi+i\in\mathcal F$. Then there are $i,i'\in\mathcal I$ with $\varphi+i,\varphi+i'\in\mathcal F$. Suppose first that $\varphi+i\in Z$ and $\varphi+i'\in Z$. Then $i-i'\in V_1$. Since $[\mathcal I]\cap V_1=\{0\}$, then $i=i'$. Thus there is at most one element $i\in\mathcal I$ with $\varphi+i\in Z$. Now there are at least $\kappa\geq 3$ elements $i\in\mathcal I$ such that $\varphi+i\in\mathcal\bigcup_{\xi<2^\mu}X_\xi$. Using the argument as in the proof of Theorem \ref{2-lin} we reach a contradiction and we obtain that $\varphi+\mathcal I\nsubseteq\mathcal F$ for every $\varphi\in V$, which means that $\mathcal A(\mathcal F)\leq\kappa+1$. Hence $\kappa\leq\mathcal A(\mathcal F)\leq\kappa+1$. 

Since $Z\subseteq\mathcal F$ and $\mathcal L(Z)=\lambda$, then $\mathcal L(\mathcal F)\geq\lambda$. From the proof of Theorem \ref{2-lin} we obtain that $\bigcup_{\xi<2^\mu}X_\xi$ does not contain 2-dimensional vector space. To show that $\mathcal L(\mathcal F)=\lambda$ it suffices to show that each 2-dimensional space $W$ contained in $\mathcal F$ must be a subset of $Z$. 

Let $W$ be a 2-dimensional space which is not contained in $Z$. In the proof of Theorem \ref{2-lin} we have shown that the cardinality of the family of one-dimensional subspaces of $W$ is less than $\mu$. Moreover, by the construction of $Z$, $W$ has at most two one-dimensional subspaces contained in $Z$. Consequently $W$ is not a subset of $\mathcal F$.  
\end{proof}

Note that $\mathcal{HL}(\mathcal F)=2$ for $\mathcal F$ constructed in the proofs of Theorem \ref{2-lin} and Theorem \ref{mainThm}.


\section{lineability of residual star-like subsets of Banach spaces} 

Let us present the following example. Let $\mathcal J$ be an ideal of subsets of some set $X$ which does not contain $X$. By $\on{add}(\mathcal J)$ we denote the cardinal number defined as $\min\{\vert\mathcal G\vert:\mathcal G\subseteq\mathcal J$ and $\bigcup\mathcal G\notin\mathcal J\}$ where $\vert\mathcal G\vert$ stands for cardinality of $\mathcal G$. Let $\mathcal N$ and $\mathcal M$ stands for $\sigma$-ideal of null and meager subsets of the real line, respectively. Then $\omega_1\leq\on{add}(\mathcal N)\leq\on{add}(\mathcal M)\leq\mathfrak c$. Moreover, if $X$ is an uncountable complete separable metric space and $\mathcal M_X$ is an ideal of meager subsets of $X$, then $\on{add}(\mathcal M)=\on{add}(\mathcal M_X)$.  

Let $V=\R$ be a linear space over $\mathbb{K}=\Q$. Let $\mathcal J$ be a translation invariant proper $\sigma$-ideal of subsets of $\R$ and let $\mathcal F$ be a $\mathcal J$-residual subset of $\R$, i.e. $\mathcal F^c\in\mathcal J$. It turns out that $\mathcal A(\mathcal F)\geq\on{add}\mathcal J$. To see it fix $F\subseteq\R$ with $\vert F\vert<\on{add}\mathcal J$. For any $f\in F$ consider a set
$$
T_f=\{t\in\R:t+f\notin\mathcal F\}=\{t\in\R:\exists g\in\mathcal F^c(t=g-f)\}\subseteq\mathcal F^c-f. 
$$
Thus $T_f\in\mathcal J$. Since $\vert F\vert<\on{add}\mathcal J$, then also $\bigcup_{f\in F}T_f\in\mathcal J$. Therefore there is $t\in\R$ such that $t+f\in\mathcal F$ for any $f\in F$. 

If we additionally assume that $\mathcal F$ is star-like, then by Theorem \ref{generalThm} we obtain that $\mathcal {HL}(\mathcal F)>\on{add}\mathcal J$. In particular if $A$ is positive Lebesgue measure (non-meager with Baire property), then $\Q A=\{qa:q\in\Q,a\in A\}$ is ($\Q$-)star-like of full measure (residual) and therefore $\mathcal {HL}(\mathcal F)>\on{add}(\mathcal N)$ ($>\on{add}(\mathcal M)$).   

Using a similar reasoning one can prove the following.
\begin{theorem}\label{residual}
Assume that $X$ is a separable Banach space. Let $\mathcal F\subseteq X$ be residual and star-like, and let $\mathbb{K}\subseteq\R$ be a field of cardinality less than $\on{add}(\mathcal M)$. Consider $X$ as a linear space over $\mathbb K$.  Then $\mathcal {HL}_{\mathbb K}(\mathcal F)>\on{add}(\mathcal M)$. In particular $\mathcal F$ contains an uncountably dimensional vector space over $\mathbb K$.    
\end{theorem}

Let $\hat{C}[0,1]$ stand for the family of functions from $C[0,1]$ which attain the maximum only at one point. Then $\hat{C}[0,1]$ is star-like and residual but not 2-lineable, see \cite{BCFPS} and \cite{GQ} for details. This shows that Theorem \ref{residual} would be false for $\mathbb K=\R$. On the other hand, Theorem \ref{residual} shows that for any field $\mathbb K\subseteq\R$ of cardinality less than $\on{add}(\mathcal M)$ there is uncountable family $\mathcal F\subseteq\hat{C}[0,1]$ such that any nontrivial linear combination of elements from $\mathcal F$ with coefficients from $\mathbb K$ attains the maximum only at one point.


\section{Homogeneous lineability and lineability numbers of some subsets of $\R^\R$}

In this section we will apply Theorem \ref{generalThm} to obtain homogeneous lineability of families of functions from $\R^\R$. We will consider those families for which the additivity has been already computed. 

Let $f\in\R^\R$. We will say that 
\begin{enumerate}
\item $f\in\text{D}(\R)$ ($f$ is Darboux) if $f$ maps connected sets onto connected sets.
\item $f\in\text{ES}(\R)$ ($f$ is everywhere surjective) if $f(U)=\R$ for every nonempty open set $U$;
\item $f\in\text{SES}(\R)$ ($f$ is strongly everywhere surjective) if $f$ takes each real value $\mathfrak c$ many times in each interval;
\item $f\in\text{PES}(\R)$ ($f$ is perfectly everywhere surjective) if $f(P)=\R$ for every perfect set $P$;
\item $f\in\text{J}(\R)$ ($f$ is Jones function) if the graph of $f$ intersects every closed subset of $\R^2$ with uncountable projection on the $x$-axis. 
\item $f\in\text{AC}(\R)$ ($f$ is almost continuous, in the sense of J. Stallings) if every open set containing the graph of $f$
contains also the graph of some continuous function.
\item If $h:X\to\R$, where $X$ is a topological space, $h\in\text{Conn}(X)$ ($h$ is a connectivity function) if the graph of $h|_C$ is connected
for every connected set $C\subseteq X$.
\item $f\in\text{Ext}(\R)$ ($f$ is extendable) if there is a connectivity function $g :\R\times[0, 1]\to\R$ such that $f(x)=g(x, 0)$ for every $x\in\R$.
\item $f\in\text{PR}(\R)$ ($f$ is a perfect road function) if for every $x\in\R$ there is a perfect set $P\subseteq\R$ such that $x$ is a bilateral limit
point of $P$ and $f\vert_P$ is continuous at $x$.
\item $f\in\text{PC}(\R)$ ($f$ is peripherally continuous) if for every $x\in\R$ and pair of open sets $U,V\subseteq\R$ such that $x\in U$ and
$f (x)\in V$ there is an open neighborhood $W$ of $x$ with $W\subseteq U$ and $f (\text{bd}(W))\subseteq V$.
\item $f\in\text{SZ}(\R)$ ($f$ is a Sierpi\'nski--Zygmund function) if $f$ is not continuous on any subset of the real line of cardinality $\mathfrak c$.
\end{enumerate} 

We start from recalling two cardinal numbers:
$$
e_\mathfrak{c}=\min\{\vert F\vert:F\subseteq\R^\R, \forall\varphi\in\R^\R\exists f\in F(\on{card}(f\cap\varphi)<\mathfrak c)\},
$$
$$
d_\mathfrak{c}=\min\{\vert F\vert:F\subseteq\R^\R, \forall\varphi\in\R^\R\exists f\in F(\on{card}(f\cap\varphi)=\mathfrak c)\}.
$$

It was proved in \cite{CM} that $\mathcal A(\text{D}(\R))=\mathcal A(\text{AC}(\R))=e_\mathfrak{c}$. Therefore $\mathcal{HL}(\text{D}(\R)),\mathcal{HL}(\text{AC}(\R))\geq e_\mathfrak{c}^+$. More recently in \cite{GMMFSS} it was proved that $\mathcal A(\text{J}(\R))=e_\mathfrak{c}$. Thus $\mathcal{HL}(\text{J}(\R))\geq e_\mathfrak{c}^+$. On the other hand, by the result of \cite{GM},  $\mathcal{L}(\text{J}(\R))=(2^\mathfrak{c})^+$. Since $\text{J}(\R)\subseteq\text{PES}(\R)\subseteq\text{SES}(\R)\subseteq\text{ES}(\R)\subseteq\text{D}(\R)$, then additivity number for the classes $\text{PES}(\R),\text{SES}(\R),\text{ES}(\R)$ is $e_\mathfrak{c}$ while their lineability number is largest possible. Since in some model of ZFC we have $e_\mathfrak{c}<2^\mathfrak{c}$, our method does not give optimal solution for lineability number in this cases.  

It was proved in \cite{CR} that $\mathcal A(\text{Ext}(\R))=\mathcal A(\text{PR}(\R))=\mathfrak{c}^+$. Thus $\mathcal{HL}(\text{Ext}(\R)),\mathcal{HL}(\text{Ext}(\R))\geq \mathfrak{c}^{++}$. These two classes were not considered in the context of lineability. 

In \cite{CN} it was proved that $\mathcal A(\text{SZ}(\R))=d_\mathfrak{c}$. Thus $\mathcal{HL}(\text{SZ}(\R))\geq d_\mathfrak{c}^{+}$. In \cite{BGPS} it was shown that $\text{SZ}(\R)$ is $\kappa$-lineable if there exists a family of cardinality $\kappa$ consisting of almost disjoint subsets of $\mathfrak c$. On the other hand in \cite{GMSS} the authors proved that if there is no family of cardinality $\kappa$ consisting of almost disjoint subsets of $\mathfrak c$, then $\text{SZ}(\R)$ is not $\kappa$-lineable. In \cite{CN} the authors proved that it is consistent with ZFC+CH and $\mathcal A(\text{SZ}(\R))=\mathfrak{c}^+<2^\mathfrak{c}$. However, if CH holds, then there is a family of cardinality $2^\mathfrak{c}$ consisting of almost disjoint subsets of $\mathfrak c$ and therefore $\text{SZ}(\R)$ is $2^\mathfrak{c}$-lineable. Consequently, as in previous examples, consistently $\mathcal A(\text{SZ}(\R))^+<\mathcal L(\text{SZ}(\R))$.

It was proved in \cite{CR} that  $\mathcal A(\text{PC}(\R))=2^\mathfrak{c}$. Therefore $\mathcal{HL}(\text{PC}(\R))= (2^\mathfrak{c})^{+}$ is the largest possible. Note that Darboux functions are peripherally continuous. Since the set of all functions which everywhere discontinuous and Darboux are strongly $2^\mathfrak{c}$-algebrable, see \cite{BGPa}, then so is $\text{PC}(\R)$, which is much stronger property that $2^\mathfrak{c}$-lineability.  

If $V$ is a vector space over $\mathbb K$ with $\vert\mathbb K\vert=\mu\geq\omega$, $\mathcal F$ is star-like, $\mathcal F\subseteq V$ and $\mathcal A(\mathcal F)=\kappa>\mu$, then $\mathcal F\cup\{0\}=\bigcup\mathcal B$, where every member $B$ of $\mathcal B$ is a linear space of dimension $\dim B\geq\kappa$. It follows from the fact that for any $f\in\mathcal F\setminus\{0\}$ and any maximal vector space $B$ contained in $\mathcal F$ such that $[f]\subseteq B$ by Theorem \ref{generalThm} we have $\dim B\geq\kappa$.  

\begin{theorem}\label{BIsLarge}
Let $\mathcal F\subseteq V$ be such that $\mathcal A(\mathcal F)=\kappa>\vert\mathbb K\vert$. Assume that there is a vector space $X\subseteq V$ such that $X\cap\mathcal F\subseteq\{0\}$ with $\dim X=\tau\leq\kappa$. Then $\mathcal B$ contains at least $\tau$ many pairwise distinct elements. 
\end{theorem}

\begin{proof} Assume first that $\dim X=\kappa$
Let $X=\bigcup_{\xi<\kappa}X_\xi$ be such that $X_\xi\subseteq X_{\xi'}$ provided $\xi<\xi'$ and $X_\xi$ is a linear space with $\dim X_\xi=\vert\xi\vert$ for every $\xi<\kappa$. Since $\vert X_\xi \vert<\mathcal A(\mathcal F)$, there is $\varphi_\xi\in\mathcal F$ with $\varphi_\xi+X_\xi\subseteq\mathcal F$. Take any two distinct elements $x,y\in X$. There is $\xi$ such that $x,y\in X_\xi$ and $\varphi_\xi+x,\varphi_\xi+y\in\mathcal F$. Thus $x-y\in X$. Hence $x,y$ are not in the same $B$ from $\mathcal B$. Consequently $\vert\mathcal B\vert\geq\kappa$. If $\dim X<\kappa$, the proof is similar and a bit simpler. 
\end{proof}

The fact that $\mathcal F\cup\{0\}$ can be represented as a union of at least $\tau$ linear spaces, each of dimension at least $\kappa$ we denote by saying that $\mathcal F$ has property $B(\kappa,\tau)$. Surprisingly families of strange function defined by non-linear properties can be written as unions of large linear spaces. 

\begin{corollary}
\begin{enumerate}
\item $\on{D}(\R)$, $\on{ES}(\R)$, $\on{SES}(\R)$, $\on{PES}(\R)$ and $\on{J}(\R)$ have $B(e_{\mathfrak c},e_{\mathfrak c})$.
\item $\on{AC}(\R)$ has $B(e_{\mathfrak c},\mathfrak c)$.
\item $\on{PC}(\R)$ has $B(2^\mathfrak{c},\mathfrak c)$.
\item $\on{PR}(\R)$ have $B(\mathfrak c^+,\mathfrak c^+)$.
\item $\on{SZ}(\R)$ has $B(d_\mathfrak{c},d_\mathfrak{c})$.
\end{enumerate} 
\end{corollary}  

\begin{proof}
We need only to show that for each of the given families of functions there is a large linear space disjoint from it. This is in fact the same as saying that complements of these families are $\tau$-lineable for an appropriate $\tau$. Note that $X_1=\{f\in\R^\R:f(x)=0$ for every $x\leq 0\}$ is disjoint from $\on{ES}(\R)$, $\on{SES}(\R)$, $\on{PES}(\R)$, $\on{J}(\R)$ and $\on{SZ}(\R)$. Let $X_2$ be a space of dimension $2^{\mathfrak c}$ such that $X_2\setminus\{0\}$ consists of nowhere continuous functions with a finite range. Such a space $X_2$ was constructed in \cite{GMMFSSContinuityMonthly}. Clearly, $X_2$ is disjoint from $\on{D}(\R)$. Let $X_3$ be a linear space of dimension $\mathfrak c$ such that $X_3\setminus\{0\}$ consists of functions which has dense set of jump discontinuities. Such a space $X_3$ was constructed in \cite{GKP}. Then $X_3\cap \on{AC}(\R)=X_3\cap \on{PC}(\R)=\{0\}$. Let $X_4$ be a linear space of dimension $2^{\mathfrak c}$ such that $X_4\setminus\{0\}\subseteq\on{PES}(\R)$. Such a space $X_4$ was constructed in \cite{GMMFSSS}. Then $X_4\cap \on{PR}(\R)=\{0\}$. 
\end{proof}

We end the paper with the list of open questions:\\
1. Is it true that $\mathcal A(\mathcal F)^+\geq\mathcal{HL}(\mathcal F)$ for any family $\mathcal F\subseteq\R^\R$?\\
2. What is are homogeneous lineability numbers of the following families $\text{D}(\R)$, $\text{AC}(\R)$, $\text{PES}(\R)$, $\text{J}(\R)$, $\text{Ext}(\R)$, $\text{PR}(\R)$ and $\text{SZ}(\R)$?\\
3. Are the families $\text{Ext}(\R)$ and $\text{PR}(\R)$ $2^\mathfrak{c}$-lineable in ZFC?\\
4. Are the complements of the families $\text{Ext}(\R)$, $\text{PC}(\R)$ and $\text{AC}(\R)$ $2^\mathfrak{c}$-lineable?


\end{document}